\documentclass[11pt]{elsarticle}

\makeatletter
\def\ps@pprintTitle{%
  \let\@oddhead\@empty
  \let\@evenhead\@empty
  \def\@oddfoot{\reset@font\hfil\thepage\hfil}
  \let\@evenfoot\@oddfoot
}
\makeatother

\usepackage[utf8]{inputenc} \usepackage[english]{babel} \usepackage{graphicx}
\usepackage{amsmath,amsthm,amssymb} \usepackage{xstring} \usepackage{ulem}
\usepackage[titletoc,title]{appendix} \usepackage{array}

\usepackage{esint}

\usepackage[pdfencoding=auto, psdextra, unicode]{hyperref}
\hypersetup{pdfborder = {0 0 0}} %Unsichtbarer Querverweis

\newcommand{\Gamto}{\stackrel{\Gamma}\longrightarrow}
\newcommand{\eps}{\epsilon}
\newcommand{\R}{\mathbb{R}}
\renewcommand{\AA}{{\mathcal A}}
\newcommand{\KK}{{\mathcal K}}
\newcommand{\FF}{{\mathcal F}}
\newcommand{\CC}{{\mathcal C}}
\renewcommand{\SS}{{\mathcal S}}
\newcommand{\N}{\mathbb{N}}
\newcommand{\Z}{\mathbb{Z}}

\newcommand{\Rn}{{\R^n}}

\DeclareMathOperator*{\dist}{dist} 

\usepackage[dvipsnames]{xcolor}

\usepackage{enumitem}

\newcommand{\T}{\mathbb{T}}

 \renewcommand{\labelenumi}{\theenumi}

\expandafter\def\csname ver@etex.sty\endcsname{3000/12/31}

\usepackage{autonum}

\newtheorem{theorem}{Theorem}[section]
\newtheorem{definition}[theorem]{Definition}
\newtheorem{proposition}[theorem]{Proposition}

\newtheorem{remark}[theorem]{Remark}
\newtheorem{lemma}[theorem]{Lemma}
\newtheorem{corollary}[theorem]{Corollary}

  { \end{itemize} }
  { \end{enumerate} }

\begin{document}

\begin{abstract}
  We derive a macroscopic limit for a sharp interface version of a model
  proposed in \cite{KoShAn-2006} to investigate pattern formation due to
  competition of chemical and mechanical forces in biomembranes. We identify
  sub-- and supercrital parameter regimes and show with the introduction of the
  autocorrelation function that the ground state energy leads to the
  isoperimetric problem in the subcritical regime, which is interpreted to not
  form fine scale patterns.
\end{abstract}


\begin{frontmatter}
  \title{\texorpdfstring{$\Gamma$--limit}{Gamma--limit} for a sharp interface model related to pattern formation
    on biomembranes}
  \author[hd]{Denis Brazke} \ead{denis.brazke@uni-heidelberg.de}
  \address[hd]{University of Heidelberg (Germany), Institute of Applied
    Mathematics}
  \author[hd]{Hans Kn\"upfer} \ead{knuepfer@uni-heidelberg.de}
  \author[hd]{Anna Marciniak--Czochra}
  \ead{anna.marciniak@iwr.uni-heidelberg.de}
\end{frontmatter}

\tableofcontents

%\paragraph{Data availability statement} Data sharing not applicable to this article as no datasets were generated or analysed during the current study.

\newpage

\section{Introduction} \label{sec:introduction}

We consider the family of nonlocal isoperimetric problems
\begin{align} \label{eq:energy} E_{\gamma,\eps}[u] := \int_{\T^n} |\nabla u| -
  \frac{\gamma}{\eps^{n + 1}} \int_{\Rn} \KK(\tfrac z\eps) \fint_{\T^n} |u(x +
  z) - u(x)| \ dx \ dz ,
\end{align}
where $u \in BV(\T^n, \{0,1\})$ and where $\T^n$, $n \geq 2$, is the
$n$--dimensional unit flat torus. Here, $\gamma > 0$ is a fixed parameter and
$\KK$ is a radially symmetric with certain integrability conditions. For details, see Section
\ref{sec:Results}. This class of problems appears in mathematical models of
complex materials (such as biomembranes, see below), diblock copolymers (see
\cite{ChoksiPeletier-2010, JulPis:2017}) and cell motility (see
\cite{AlMeMe:2020}). We derive the limit $\eps \to 0$ of the family of models
$E_{\gamma,\eps}$ in the framework of $\Gamma$--convergence.

\medskip

The family of functionals \eqref{eq:energy} is a more general sharp interface
version of a model proposed in \cite{KoShAn-2006} for the investigation of
structures which arise due to competing diffusive and mechanical forces. In
particular, it is relevant for the modelling of formation of so called
\textit{lipid rafts} in cell membranes. These are complex nanostructures made up
of lipids, proteins and cholesterol and are believed to be responsible for many
biological phenomena such as transmembrane signaling and cellular homeostasis
(see e.g. \cite{SimonsIkonen1997, LorLev:2015, RajendranSimons1099,
  SonPri:2013}). A diffuse interface version of \eqref{eq:energy} was
considered in \cite{FHLZ-2016} and can be formulated in terms of the order
parameter (i.e. concentration of the chemical) $u$ alone. The resulting energy
is given by
\begin{align} \label{eq:parameter_energy} \FF_{q,\eps}[u] = \frac 1\eps
  \int_{\Omega} W(u) + (1 - q) \eps^2 |\nabla u|^2 - u^2 + u (\mathbf 1 - \eps^2
  \Delta)^{-1}u \ dx,
\end{align}
where $\Omega$ is a reference domain, $W$ is a double well potential, $\eps > 0$
and $q < 1$. The solution operator $(\mathbf 1 - \eps^2\Delta)^{-1}$ is subject
to Neumann boundary conditions. In \cite{FHLZ-2016} it was shown that for
sufficiently small $q > 0$, the family $\FF_{q,\eps}$ $\Gamma$--converges to a
constant multiple of the perimeter functional as $\eps \to 0$ in the
$L^2$--topology for a standard family of double well potentials (quadratic roots
and quadratic growth). We comment on the relation between the diffuse and sharp
interface model in Section \ref{sec:Results} below.

\medskip

In this article, we compute explicitely the $\Gamma$--limit of the family
$E_{\gamma,\eps}$ under suitable assumptions on $\KK$. We identify two parameter
regimes (sub-- and supercritical) with respect to $\gamma$ and show that the
limit problem is the isoperimetric problem (with modified prefactor) in the subcritical regime. In the
supercritical regime, we show that minimising sequences of $E_{\gamma,\eps}$
always have unbounded perimeter. 

\medskip

Our main observation is that both the perimeter and the non--local term in
\eqref{eq:energy} can be represented in terms of the (symmetrised) autocorrelation function
\begin{align}
  c_u(r) := \fint_{\mathbb S^{n - 1}} \int_{\T^n} u(x + rw) \, u(x) \ dx \ dw.
\end{align}
The formulation of the energy in terms of the autocorrelation function (Lemma
\ref{lem:reformulation_non_local}) reveals that although the energy
$E_{\gamma,\eps}$ is not linear in $u$, it is linear in terms of the
autocorrelation function. For the proofs of the upper and lower bound we then
can use properties of the autocorrelation function, which allows us to treat a
large class of kernels.

\paragraph{Related Literature}

Different classes of kernels in isoperimetric problems, where the
non--local term has the same scaling as the local one, were considered e.g. in
\cite{CesNov:2020, MelWu:2022, MurSim:2019, RenWei-2000}. Representing the
energy in terms of the autocorrelation function and exploiting its fine
properties has been proposed for a similar family of energies in
\cite{KnuShi:2021}. There, the autocorrelation function was used to derive a
second order expansion for the perimeter functional. However, the family of
kernels is subject to different conditions and converge monotonically to a
measurable function with certain integrability properties, wheras our family of
kernels is assumed to form an approximation of the identity. Due to the highly
singular nature of the considered kernels near the origin, our techniques require more information about
the regularity of the autocorrelation function. This is not needed in
\cite{KnuShi:2021} by the assumption of monotone convergence of the kernels. For
this purpose, we show regularity properties of the autocorrelation function for
polytopes, which are able to approximate any shape of finite perimeter (see \eqref{eq:density}).

\medskip

The $\Gamma$--convergence of \eqref{eq:energy} where $\KK$ is the solution to
the Helmholtz equation with different boundary conditions was considered in
\cite{MelWu:2022}. Their techniques are based on PDE arguments and they require
the family of kernels to be solutions of the Helmholtz equation, wheras our
techniques work for a larger class of kernels and are based on decay
estimates. In particular, we recover the $\Gamma$--convergence result and are
able to extend it to kernels which are given by
$\KK + (-\Delta)^{\frac s2} \KK = \delta_0$ for any $1 < s \leq 2$ (see the
appendix).

\medskip

The non--local term appearing in \eqref{eq:energy} is a periodic version of the
quantity already considered in the famous work \cite{BBM:2001} and was further
investigated in \cite{Davil:2002}. There it was shown that the non--local term
converges pointwise to a constant multiple of the perimeter functional (under
suitable assumptions on the family of
kernels).
Our main result will deal with $\Gamma$--convergence, which is in general 
different to pointwise convergence (see e.g. \cite[Example 4.4]{Dal-1993}).

\medskip

Lastly, non--local isoperimetric problems of the form \eqref{eq:energy} have
also been considered in the mathematical literature in many different models,
e.g. sharp interface variants of the Ohta-Kawasaki model where the non--local
term is a Coulomb or Riesz type interaction (see \cite{AcerbiFuscoMorini-2011,
  AlbertiChoksiOtto-2009, ChoksiPeletier-2010, CicSpa:2013, Cristo:2015,
  GoMuSe:2013, GoMuSe:2014, JulPis:2017, MorSte:2014}). We note that, contrary
to our energy, in these models the non--local term does not have the same
scaling as the perimeter functional.

\paragraph{Structure of the paper}

In Section \ref{sec:Results} we present and discuss our main results. We also
outline the proofs. In Section \ref{sec:Autocorrelation} we introduce the
autocorrelation function and give some of its properties. In Section
\ref{sec:proofs} we reformulate the energy in terms of the autocorrelation
function and present the proofs of our main results, namely the compactness and
$\Gamma$--convergence of $E_{\gamma,\eps}$ (see Theorem \ref{thm:compactness}
and Theorem \ref{thm:convergence}). In the Appendix we collect some further calculations.

\paragraph{Notation} \label{subsec:notation} Throughout this paper, we write
$\T^n := \R^n/\Z^n$ for the $n$ dimensional flat torus. Without distinction in
notation, we identify $(0,1)^n$--periodic functions in $\R^n$ with functions on
$\T^n$. We denote $|\Omega|$ for the Lebesgue measure of $\Omega$. For $r > 0$
and $x \in \R^n$ we denote $B_r(x) := \{y \in \R^n: |x - y| < r\}$ for the ball
around $x$ of radius $r$ and the unit sphere by
$\mathbb S^{n - 1} := \partial B_1(0)$. We denote $\omega_n$ for the volume of
the unit ball in $\R^n$ and $\sigma_n := n \omega_n$ for its surface area. We
write $A \lesssim B$ if there exists a universal constant $C > 0$, such that
$A \leq C B$. A function $u \in L^1(\T^n)$ is said to be a function of bounded
variation if {\small \begin{align} \int_{\T^n} |\nabla u| := \sup
    \bigg\{\int_{\T^n} u(x) \, (\nabla \cdot \varphi)(x) \ dx : \varphi \in
    C^1(\T^n, \R^n), \ \|\varphi\|_{L^\infty(\T^n)} \leq 1 \bigg\} < \infty.
	\end{align}}%
      In this case, we write $u \in BV(\T^n)$. We denote the perimeter functional $P : L^1(\T^n) \longrightarrow [0,\infty]$ by
      \begin{align}
        P[u] := \left\{ \begin{array}{lcl} 
                          \displaystyle{\int_{\T^n} |\nabla u|} && \text{if } u \in BV(\T^n,\{0,1\}), \\[10pt] 
                          +\infty && \text{if } u \in L^1(\T^n) \setminus BV(\T^n,\{0,1\}).
                        \end{array} \right.
      \end{align}

      \paragraph{Acknowledgments}

      D.B. would like to warmly thank M. Mercker, T. Stiehl, J. Fabiszisky,
      B. Brietzke, C. Tissot and A. Tribuzio for valuable discussions on the topic. This
      work is funded by \textit{Deutsche Forschungsgemeinschaft} (DFG, German
      Research Foundation) under Germany's Excellence Strategy
      EXC-2181/1-39090098 (the Heidelberg STRUCTURES Cluster of Excellence).

      \section{Statement and discussion of the main results} \label{sec:Results}

      In this section, we give the precise
      formulation of our setting and main results. Throughout this article, we
      assume $n \in \N$, $n \geq 2$. Since we are interested in the case of
      prescribed volume fraction, we introduce the class of admissible functions
      \begin{align}
        \AA := \bigg\{u \in BV(\T^n,\{0,1\}) : \int_{\T^n} u(x) \ dx \ = \ \theta \bigg\}
      \end{align}
      for some fixed parameter $0 < \theta < 1$. We also fix the function
      $\KK : \R^n \longrightarrow \R$ and define $\KK_\eps(z) := \frac 1{\eps^n} \KK(\frac z\eps)$.
      We recall the energy functional
      \begin{align}\label{eq:definition_energy}
        E_{\gamma,\eps}[u] = P[u] - \frac \gamma \eps \int_{\R^n} \KK_\eps(z) \int_{\T^n} |u(x + z) - u(x)| \ dx \ dz
      \end{align}
      when $u \in \AA$, and $E_{\gamma,\eps}[u] = +\infty$ for all
      $u \in L^1(\T^n) \setminus \AA$. The kernel $\KK$ is subject to the following conditions:
      {\renewcommand{\labelenumi}{\arabic{enumi}}
	\begin{enumerate} \itemsep 6pt 
        \item[(H1)] \label{it-H1} $\KK$ is radial.
        \item[(H2)] \label{it-H2} $ |z| \, \KK(z) \in L^1(\Rn)$.
        \item[(H3)] \label{it-H3} $\displaystyle \int_{B_r(0)^c} \KK(z) \ dz \geq 0$ for all $r > 0$.
        \end{enumerate}}              
      \newcommand{\hyp}[1]{{\rm (\hyperref[it-H#1]{H#1})}}              
     By a
      slight abuse of notation we also write $\KK(z) = \KK(|z|)$. The
      imposed conditions \hyp1--\hyp3 are different to the conditions
              imposed on the family of kernels considered in \cite{Davil:2002}
              (see also \cite[Theorem 1.2]{MelWu:2022}): although our family of
              kernels arises as $L^1$--dilates of $\KK$, we relax the assumption
              of the non--negativity of $\KK_\eps$ by \hyp3. 

              \medskip

              As the quantity $E_{\gamma,\eps}$ is the difference of two
              positive terms with the same scaling, uniform control of the
              energy of a sequence of shapes in general does not a priori lead
              to control of the perimeter. However, we identify the critical
              parameter
              \begin{align} \label{def-gam} %
		\gamma_{\rm crit} := \frac{n \omega_n}{2\omega_{n - 1}} \bigg( \int_{\Rn} |z| \, \KK(z) \ dz \bigg)^{-1} 
              \end{align}
              such that for $0 < \gamma < \gamma_{\rm crit}$ the perimeter is
              controlled by the energy, while the perimeter is not controled for
              $\gamma > \gamma_{\rm crit}$.

              \begin{theorem}[Compactness and
                non--compactness] \label{thm:compactness} Let $n \geq 2$ and
                suppose that $\KK$ satisfies \hyp1 -- \hyp3.
                \begin{enumerate}
                \item \label{it-thm:compactness:compactness} Compactness: Let
                  $0 \leq \gamma < \gamma_{\rm crit}$. Then for any sequence
                  $u_\eps \in L^1(\T^n)$ with
                  $\sup_{\eps} E_{\gamma,\eps}[u_\eps] < \infty$ there exists
                  $u \in \AA$ and a subsequence (not relabeled) such that
                  \begin{align}
                    u_\eps \xrightarrow{\eps \to 0} u	&& \text{in } L^1(\T^n).
                  \end{align}
                \item \label{it-thm:compactness:non_compactness}
                  Non--compactness: Let $\gamma \geq \gamma_{\rm crit}$. Then
                  there exists a sequence $u_\eps \in \AA$ such that
                  $\sup_{\eps} E_{\gamma,\eps}[u_\eps] < \infty$, but there is
                  no $u \in \AA$ such that $u_\eps \to u$ in $L^1(\T^n)$ for any
                  subsequence.
		\end{enumerate}
              \end{theorem}
              Theorem \ref{thm:compactness} identifies the subcritical parameter
              regime $0 \leq \gamma < \gamma_{\rm crit}$ and the supercritical
              parameter regime $\gamma \geq \gamma_{\rm crit}$. We learn that
              the local term dominates in the subcritical regime, while in the
              supercritical regime the destabilising effect of the non--local
              interaction takes over. The failure of compactness is reflected in
              uniform lower bounds of the energy and the behaviour of minimising
              sequences:

              \begin{corollary}[Lower bound of the
                energy] \label{coro:finer_behaviour} Let $n \geq 2$ and suppose
                that $\KK$ satisfies \hyp1 -- \hyp3.
                \begin{enumerate}
		\item \label{it-coro:finer_behaviour1} Let
                  $0 \leq \gamma \leq \gamma_{\rm crit}$. Then
                  $E_{\gamma,\eps} \geq 0$ for every $\eps > 0$.
		\item \label{it-coro:finer_behaviour2} Let $\gamma > \gamma_{\rm
                    crit}$. Then there exists a sequence $u_\eps \in \AA$ such
                  that
                  \begin{align}\label{eq:sequence_energy_negative_infinity}
                    E_{\gamma,\eps}[u_\eps] \xrightarrow{\eps \to 0} -\infty.
                  \end{align}
                  Moreover
                  $\displaystyle{\liminf_{\eps \to 0} P[u_\eps] = \infty}$ for
                  any sequence, such that
                  \eqref{eq:sequence_energy_negative_infinity} holds.
                \end{enumerate}
              \end{corollary}
              We note that in the critical case $\gamma = \gamma_{\rm crit}$,
              the family of energies $E_{\gamma_{\rm crit},\eps}$ does not have
              the compactness property even though it is bounded from
              below. This is due to the fact that
              $0 \leq E_{\gamma_{\rm crit},\eps}[u] \to 0$ for polytopes (see Proposition \ref{prp-pointwise_bound_energy} and
              \ref{prp-energy_convergence_polytopes}), which are able to approximate any shape of finite perimeter (see \eqref{eq:density}). In this case, higher
              order effects play a role and the theory of higher order
              $\Gamma$--expansions, developed in \cite{AnzBal:1993}, should be
              suitable.

              \medskip

              As explained in the introduction, we are interested in the
              limiting behaviour of minimising sequences as $\eps \to 0$. We
              show that our family of non--local energies $\Gamma$--converges to
              a local energy functional:
              \begin{theorem}[$\Gamma$--convergence] \label{thm:convergence} Let
                $n \geq 2$ and suppose that $\KK$ satisfies \hyp1 -- \hyp3. Let
                $0 \leq \gamma \leq \gamma_{\rm crit}$. Then
                $E_{\gamma,\eps} \stackrel{\Gamma}{\longrightarrow}
                E_{\gamma,0}$ with respect to the $L^1$--topology, where
                \begin{align}\label{eq:definition_limit_energy}
                  E_{\gamma,0}[u] := \left\{ \begin{array}{lcl} (1 - \frac{\gamma}{\gamma_{\rm crit}}) P[u] && \text{if } u \in \AA, \\[10pt] 
                                               +\infty											&& \text{if } u \in L^1(\T^n) \setminus \AA.
                                             \end{array} \right.
		\end{align}
                In particular, for every $u \in L^1(\T^n)$ we have:
		\begin{enumerate}
                \item Liminf inequality: For every sequence
                  $u_\eps \in L^1(\T^n)$ such that $u_\eps \to u$ in $L^1(\T^n)$
                  we have
                  $E_{\gamma,0}[u] \leq \liminf_{\eps \to 0}
                  E_{\gamma,\eps}[u_\eps]$.
                \item Limsup inequalty: There is a sequence
                  $u_\eps \in L^1(\T^n)$ such that $u_\eps \to u$ in $L^1(\T^n)$
                  and
                  $\limsup_{\eps \to 0} E_{\gamma,\eps}[u_\eps] \leq
                  E_{\gamma,0}[u]$.
		\end{enumerate}
              \end{theorem}
%              We remark that the liminf--inequality (i) above does not require
%              the assumption \hyp0. 
              Although the non--local term leads to a
              reduction of the interfacial cost by a factor of
              $\frac{\gamma}{\gamma_{\rm crit}}$ in the subcritical regime, the
              limit problem is simply the isoperimetric problem with a suitably 
              modified prefactor. Hence, the unique minimiser
              (up to translation and rotation) is given by $u = \chi_\Omega$,
              where $\Omega \subset \T^n$ is a single ball or the complement of
              a single ball if
                \begin{align}
                  \min\{\theta, 1 - \theta\} <  \Big(\frac 2n \Big)^{\frac{n}{n + 1}} \omega_n^{-\frac{1}{n - 1}},
                \end{align}
                and by a single laminate otherwise. If equality holds, both ball
                and laminate are solutions.

                \medskip

                Since our admissible class of functions is restricted to
                functions with values in $\{0,1\}$, both main theorems hold true
                with respect to the $L^p$--topology for any $p \in
                [1,\infty)$. Furthermore, we note that the main theorems also
                hold without the assumption of a fixed volume fraction,
                i.e. for admissible class of functions $\AA = BV(\T^n,\{0,1\})$
                as can be shown with minor modifications of the proofs.

                \paragraph{Strategy of the proofs}
                In the following we give an overview of the central strategy
                of the proofs for the theorems, noting that the detailed proofs
                are given in Section \ref{subsec:proofs}. Central to our proofs
                is the formulation of the energy in terms of the symmetrised
                autocorrelation function (see Lemma
                \ref{lem:reformulation_non_local}). At this stage, the radial
                symmetry of the kernel is crucial and the prefactor $\eps^{-1}$
                of the non--local term ensures that the weight in the
                reformulation is an approximation of the identity. This implies
                that the limiting behaviour of the energy only depends on
                the autocorrelation function near the origin.

              \medskip

              The lower bound in Theorem \ref{thm:convergence} follows from
              sharp estimates of the auto--correlation function. The upper bound
              requires higher regularity of the autocorrelation function near
              the origin in order to pass to the limit. For this purpose, we
              show that the autocorrelation function is a polynomial
              near the origin for polytopes (see Lemma
              \ref{lemma:polytope_autocorrelation}). The reformulation of the
              energy in terms of the autocorrelation function allows for a
              splitting into behaviour near the origin (which is regular for
              polytopes as the autocorrelation function is a polynomial) and far
              from the origin (which is small since the weight is an
              approximation of the identity). The remainder of the proof of the
              upper bound follows from approximation with polytopes (see \eqref{eq:density}). 
	
              \paragraph{Relation to biological model and comparision of sharp and diffuse interface models}
              As
              explained in the introduction, the underlying biological model is
              a diffuse interface model (see \eqref{eq:parameter_energy}).
              In the
              Modica--Mortola theory, there is a connection between diffuse and
              sharp interface energies in the framework of $\Gamma$--convergence, namely (see 
              \cite{ChoksiSternberg2006,Modica-1987a})
              \begin{align} \label{eq:diffuse_sharp_explanation} %                 
                  \int_{\T^n}\frac 1\eps W(u) + (1 - q) \eps |\nabla u|^2 \ dx
                  \Gamto c_{W,q} \int_{\T^n} |\nabla u|
              \end{align}
              as $\eps \to 0$, where
              \begin{align} \label{eq:sharp_interface_constant}
              	 c_{W,q} = 2 \sqrt{1 - q} \int_0^1 \sqrt{W(s)} \ ds.
              \end{align}
              
              Naively, this would suggest to correspondingly replace the diffuse
              interface term in \eqref{eq:parameter_energy} by its sharp interface counterpart as in \eqref{eq:diffuse_sharp_explanation}.
              In particular,
              we note that for $\KK$ being the solution to
              	\begin{align} \label{eq:helmholtz_kernel}
              		\KK - \Delta \KK = \delta_0 	&& \text{in } \Rn,
              	\end{align}
                $\FF_{q,\eps}$ in \eqref{eq:parameter_energy} is a corresponding diffuse interface model of $E_{\gamma,\eps}$ as described above. The energy $\FF_{q,\eps}$ has been considered in
                \cite{FHLZ-2016} where it was
                shown that there exists $\bar q > 0$, such that $\FF_{q,\eps}$
                $\Gamma$--converges to a constant multiple of the perimeter
                functional. We note that the results of \cite{FHLZ-2016} do not
                contain a characterisation or estimation of $\bar q$, nor the
                constant appearing in their $\Gamma$--limit. This is due to a
                nonlinear interpolation inequality shown in
                \cite{ChDaMaFoLe:2011}, which has no explicit description of the
                constants involved (see also \cite{CiSpNuZe:2011}).   On the other hand our
                result %for the corresponding Helmholtz kernel
                shows that the critical value $q_{\rm crit}$ of the
                corresponding sharp interface model (using the relations \eqref{eq:diffuse_sharp_explanation} and \eqref{eq:sharp_interface_constant}) is given by
            		\begin{align}
            			q_{\rm crit} = 1 - \Big(4 \gamma_{\rm crit} \int_0^1 \sqrt{W(s)} \ ds\Big)^{-2}.
            		\end{align}
          		For $\KK$ as in \eqref{eq:helmholtz_kernel}, we can explicitely compute the corresponding
                critical value of the sharp interface model and it is given by
                $\gamma_{\rm crit} = 1$ (Lemma \ref{lem:helmholtz}). However, one cannot equate the parameter $q_{\rm crit}$ with the corresponding critical value $\bar q$ in \cite{FHLZ-2016}. 
              Indeed, for a suitable doublewell potential with $\|W\|_{L^\infty(0,1)} < \frac 1{16}$ we would get $q_{\rm crit} < 0$. This shows, that $\bar q = q_{\rm crit}$ cannot hold, since the result of \cite{FHLZ-2016} shows that $\bar q > 0$.
		This is
        probably due to the fact that the sharp interface model neglects effects
        which might occur of properties of the non--linearity (i.e. the double
        well potential $W$) such as asymmetry, compressibility of the chemical
        substance, or behaviour at infinity or the roots. In particular, the
        interpolation inequality used in \cite{FHLZ-2016} in order to show
        $\bar q > 0$ only works for double well potentials $W$ which have
        quadratic roots and superquadratic growth (see also \cite[Section
        3]{CiSpNuZe:2011}). The sharp interface model does not contain
        information about the double well potential $W$ other than the quantity
        $\int_0^1 \sqrt{W(s)} \ ds$, which is unrelated to the aformentioned
        properties.
            
\medskip
              
              Our main result is similar to that of \cite{FHLZ-2016} in that we
              find two regimes for a parameter related to the relative strength
              of the non--local interaction (for our energy $\gamma$, in
              \cite{FHLZ-2016} it is $q$), and the limit problem in the
              subcritical regime is the isoperimetric problem. The
              interpretation is the same as in \cite{FHLZ-2016}, namely in the
              subcritical parameter regime, fine scale pattern formation does
              not occur in contrast to experimental results (e.g. stripe
              patterns or hexagonal array of balls, see
              \cite{RoKaGr:2005, KaiGro:2010}). We conjecture fine scale patterns
              to form for the critical value $\gamma = \gamma_{\rm crit}$, which
              is subject of current research.

              \medskip

\medskip

Although the reformulation of the non--local part in terms of the autocorre--lation function (see Lemma \ref{lem:reformulation_non_local}) is also valid for the diffuse interface model \eqref{eq:parameter_energy}, our methods used for the sharp interface model \eqref{eq:definition_energy} fail for two main reasons. First, the autocorrelation function lacks useful properties for Sobolev functions $u \in H^1(\T^n)$ such as $c'_u(0) = 0$ in contrast to Proposition \ref{prp-autocorrelation_bv} \ref{it-prp-autocorrelation_bv:perimeter}. Second, the energy is not anymore solely representable by the autocorrelation function due to the non--linearity $W$, i.e. the diffuse interface model is not linear in the space of autocorrelation functions, which our methods heavily rely on.
	
\begin{remark}[Domains with arbitrary periodicity] %
  The choice of the unit flat torus in contrast to flat tori with side length
  $\ell > 0$ is justified by the scaling of the energy with respect to the
  reference domain. More precisely, let $\T^n_\ell := \R^n/(\ell \Z)^n$ and
  define
\begin{align}
  E_{\gamma,\eps}^{(\ell)}[u] := \int_{\T^n_\ell} |\nabla u| - \frac{\gamma}{\eps^{n + 1}} \int_{\R^n} \KK(\tfrac z\eps) \int_{\T^n_\ell} \frac{|u(x + z) - u(x)|}{|\T^n_\ell|} \ dx \ dz.
\end{align}
Let $u \in BV(\T^n,\{0,1\})$ and define the rescaled function $u_\ell \in BV(\T_\ell^n,\{0,1\})$ via $u_\ell(x) := u(\frac x\ell)$. Then it holds
\begin{align} \label{eq:energy_scaling}
  E^{(\ell)}_{\gamma,\eps}[u_\ell] = \ell^{n - 1} E_{\gamma,\ell^{-1} \eps}[u].
\end{align}
Thus, the limiting behaviour $\eps \to 0$ as well as the subsequent analysis can
be reduced to the case of the unit flat torus and the citical value
$\gamma_{\rm crit}$ is independent of the choice of $\ell$.
\end{remark}

\section{Autocorrelation function} \label{sec:Autocorrelation}

We introduce the radially symmetrised autocorrelation function.
\begin{definition}[Autocorrelation function]
	Let $u \in L^1(\T^n)$. We define the autocorrelation function $C_u: \R^n \longrightarrow \R$ by
        \begin{align}
          C_u(z) := \int_{\T^n} u(x + z) \, u(x) \ dx.
        \end{align}
	We also define its radially symmetrised version $c_u: \R_+
        \longrightarrow \R$ by
        \begin{align}
          c_u(r) := \frac{1}{\sigma_n} \int_{\mathbb{S}^{n - 1}} C_u(rw) \ dw.
        \end{align}
      \end{definition}

      Since $C_u$ is $(0,1)^n$--periodic, we can identify $C_u$ with a function
      on the flat torus $\T^n$. Note that $c_u$ is generally not periodic.

      \medskip

      The autocorrelation function was studied in the whole space setting (see
      \cite{Matheron:1986, Galerne:2011}). It is linked to the covariogram of a
      set $\Omega \subset \R^n$ which is defined by
      \begin{align}\label{eq:covariogram_full_space}
        g_\Omega(x) := |\Omega \cap (\Omega + x)| = \int_{\R^n} u(z) \, u(x + z) \ dz,	&& \text{where } u = \chi_\Omega.
      \end{align}
      In \cite{Galerne:2011} it was shown that $g_\Omega$ is a Lipschitz
      function if and only if $\Omega$ is a set of finite perimeter. The author
      also presents formulas for the perimeter of $\Omega$ in terms of the
      covariogram, which we will also exploit in this work (adjusted for subsets
      of the flat torus $\T^n$, see also \cite{KnuShi:2021}).

      \medskip

      The covariogram has also been investigated in the context of convex
      geometry (see e.g. \cite{MeReSc1993, Nagel:1993, GenBia:2015}). In
      \cite{MeReSc1993} it was shown that under certain smoothness conditions of
      the boundary of a set, the covariogram of a convex body is a smooth and
      they also present formulas for the first and second derivatives in this
      case. Further applications are random sets \cite{Galerne:2011,
        Bianchi:2002, ClaMal:1970} and micromagnetism \cite{KnuShi:2021}.

      \medskip

      We briefly recall basic properties of the autocorrelation
      function:

\begin{proposition}\label{prp-autocorrelation_bv}
  Let $u \in BV(\T^n,\{0,1\})$. Then $C_u$ and $c_u$ are Lipschitz continuous and it holds:
  \begin{enumerate} \itemsep 6pt
  \item \label{it-prp-autocorrelation_bv:origin}
    $0 \leq C_u(x) \leq C_u(0) = \|u\|_{L^1(\T^n)} \qquad$ for all $x \in \T^n$.
  \item \label{it-prp-autocorrelation_bv:origin2}
    $0 \leq c_u(r) \leq c_u(0) = \|u\|_{L^1(\T^n)} \qquad$ for all $r > 0$.
  \item \label{it-prp-autocorrelation_bv:perimeter}
    $\displaystyle{ \|c_u'\|_{L^\infty(0,\infty)} = -c_u'(0) = \frac{\omega_{n -
          1}}{n \omega_n} \int_{\T^n} |\nabla u|}$.
	  \end{enumerate}
\end{proposition}

\begin{proof} See \cite[Proposition 2.3 and Proposition 2.4]{KnuShi:2021}.\end{proof}

For the proof of the upper bound of Theorem \ref{thm:convergence}, in order to pass to the limit $\eps \to 0$, we need higher regularity than Lipschitz continuity of the autocorrelation function near the origin. However, one cannot expect higher regularity in general since e.g. $c''_{\chi_\Omega}(0) = +\infty$ when $\Omega$ has a cusp. We will work around this issue by showing that the autocorrelation function of polytopes is a polynomial near the origin, and hence regular. The desired limit for a general shape of finite perimeter will be dealt with via approximation by polytopes, i.e. for every $u \in BV(\T^n,\{0,1\})$ there exists a sequence of polytopes $\Omega_k \subset \T^n$ such that (see \cite{Fed58})
	\begin{align} \label{eq:density}
		\|u - \chi_{\Omega_k}\|_{L^1(\T^n)} \xrightarrow{k \to \infty} 0,	&& |P[\chi_{\Omega_k}] - P[u]| \xrightarrow{k \to \infty} 0.
	\end{align}
The following result is known in the full space setting for convex polytopes for $n = 2$ (see \cite{Nagel:1993,GarZha:1998}), but we were not able to find versions of it for arbitrary polytopes. We note that a closed set $\Omega \subset \T^n$ is said to be a polytope, if the restriction of its canonical embedding to $[0,1]^n$ is a polytope in $\R^n$. 

\begin{lemma}[Autocorrelation function for polytopes] \label{lemma:polytope_autocorrelation}
  Let $\Omega \subset \T^n$ be a polytope and let $u := \chi_\Omega$. Then there
  exists $0 < R < 1$ and $a_2,\ldots,a_n \in \R$
  with $a_2 \geq 0$ such that
		\begin{align} \label{eq:ac} %
			c_u(r) = \|u\|_{L^1(\T^n)} - r\frac{\omega_{n - 1}}{n \omega_n} \int_{\T^n} |\nabla u| + \sum_{j = 2}^n a_j r^j	&& \text{for all } 0 \leq r < R.
		\end{align}
\end{lemma}

	\begin{figure}[t] \begin{center}
		\includegraphics[scale=.18]{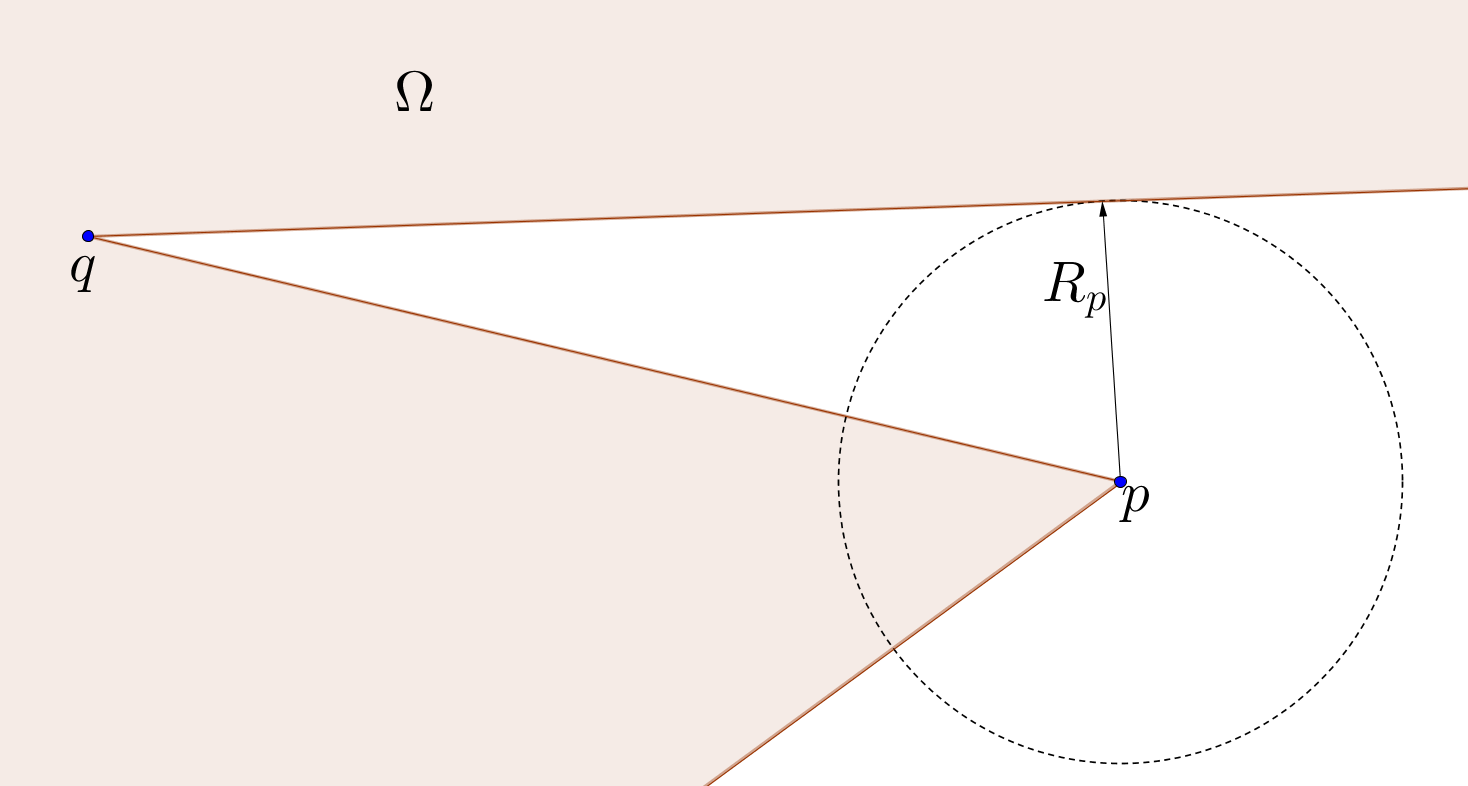}
		\caption{The quantity $R$ ensures that the corners of $\Omega \bigtriangleup \Omega_{rw}$ do not intersect faces which are not adjacent to them.} \label{fig:r_p_explanation}
              \end{center} \end{figure}

\begin{proof}
  In \cite[Lemma 2.5]{KnuShi:2021} it was shown that \eqref{eq:ac} holds for
  stripes with $a_2, \ldots, a_n = 0$. Since the only polytopes which do not
  have corners are stripes, we can without loss of generality assume that
  $\Omega$ has at least one corner. We denote the set of corners of $\Omega$ by
  $\CC$. We define
  \begin{align} \label{def-Rpoly} %
    R := \frac 12 \min \Big\{ \min_{p \in \CC} R_p,
    \min_{\substack{p,q \in \CC \\ p \neq q}} \dist(p, q) \Big\}  %
    >  0,
  \end{align}
  where for any $p \in \CC$ we set
  \begin{align}
    R_p	& := \sup \big\{ \varrho > 0 : \substack{B_\varrho(p) \cap \Omega \text{ and } B_\varrho(p) \cap \Omega^c \text{ have}\\ \text{only one connected component}} \big\} > 0.
  \end{align}
Let $r > 0$,
$w \in \mathbb S^{n - 1}$ and let $\Omega_{rw} := \Omega + rw$. Then
  \begin{align} \label{eq:covariogram_difference} C_u(0) - C_u(rw) =
    |\Omega_{rw} \setminus \Omega| = |\Omega \setminus \Omega_{rw}| = \frac 12
    |\Omega \bigtriangleup \Omega_{rw}|,
  \end{align}
  where $\bigtriangleup$ denotes the symmetric difference.
	Since $C_u(0) = \|u\|_{L^1(\T^n)} = |\Omega|$ by Proposition
    \ref{prp-autocorrelation_bv} \ref{it-prp-autocorrelation_bv:origin} it
    hence remains to calculate $|\Omega \bigtriangleup \Omega_{rw}|$ and its
    dependence on $r$.	

  \medskip

  We assume that $w \not\in \SS$, where $\SS \subset \mathbb S^{n-1}$ is the negligible 
    set of vectors which are tangential to any of the (finitely many) faces of
    $\Omega$.  By the choice of $R$, the number of faces $N$ of
  $\Omega \bigtriangleup \Omega_{rw}$ is then is independent of $r \in (0,R)$
  and independent of $w$ (see Figure \ref{fig:r_p_explanation}).  Furthermore, 
  there exist convex  trapezoids
  $T_{rw}^{(1)},\ldots T_{rw}^{(N)} \subset \T^n$ with pairwise disjoint
  interiors such that (see Figure \ref{fig:trapeziums})
  \begin{align} \label{eq:covariogram_trapeziums} \Omega \bigtriangleup
    \Omega_{rw} = \bigcup_{i = 1}^N T_{rw}^{(i)} && \text{for all } 0 < r < R, \ w
    \in \mathbb S^{n - 1} \setminus \SS.
  \end{align}
  Each of these convex trapezoids
  $T_{rw}^{(i)}$, $i \in \{1,\ldots, N\}$ is determined by a system of linear
  inequalities
  \begin{align}\label{eq:equation_trapeziums}
    T_{rw}^{(i)} = \{x \in \T^n : A^{(i)}x \leq b^{(i)}, \ A^{(i)}[x + rw] \leq b^{(i)}\},
  \end{align}
  where for $x,y \in \T^n$ we denote $x \leq y$ if and only if
  $x_j \leq y_j \! \mod 1$ for every $j \in \{1,\ldots,n\}$ and for some matrices
  $A^{(i)} \in \R^{2N \times n}$ and vectors $b^{(i)} \in \T^{2N}$. By a slicing argument there exist
$d_1^{(i)}(w), \ldots, d_n^{(i)}(w) \in \R$ such that (see Fig.~\ref{fig:trapezium_order})
  \begin{align} \label{eq:trapezium_formula} |T_{rw}^{(i)}| = \sum_{k = 1}^n
    d_k^{(i)} \! (w) \, r^k.
  \end{align}

    \begin{figure}[t] \begin{center} \includegraphics[width =
      \linewidth]{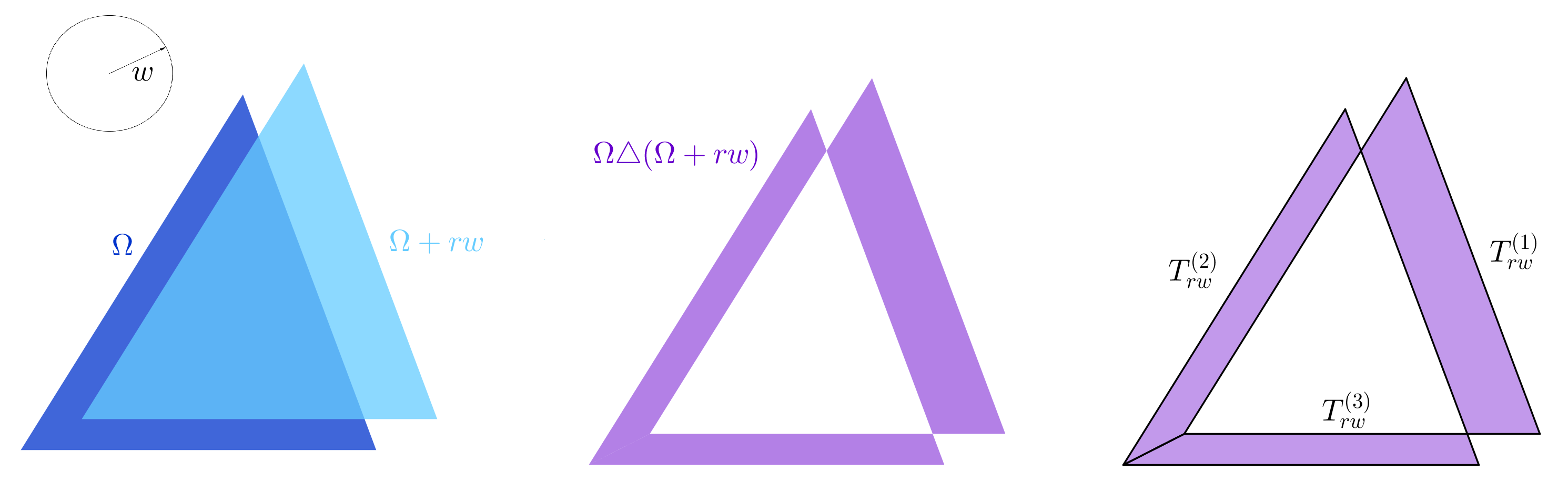}
      \caption{Construction of the trapezoids $T_{rw}^{(i)}$.} \label{fig:trapeziums}
    \end{center} \end{figure}

More precisely, for $x = (x_1,x')$, there exist at most $2N$ intervals $I_{m_1} = [f_{m_1}(x', rw), g_{m_1}(x',rw)]$, where $f_{m_1}$ and $g_{m_1}$ depend in an affine way on $x'$ and $rw$, and convex polytopes $K_{m_1} \subset \T^{n - 1}$ such that
	\begin{align}
			|T_{rw}^{i}| = \sum_{m_1 = 1}^{2N} \int_{K_{m_1}} \int_{I_{m_1}} \ dx_1 \ dx' =  \sum_{m_1 = 1}^{2N} \int_{K_{m_1}} g_{m_1}(x',rw) - f_{m_1}(x',rw)\ dx'.
		\end{align}
	The integrand is now a polynomial in $r$ of degree $1$ and the set of integration is again a convex polytope of the form \eqref{eq:equation_trapeziums} (reduced in dimension, with different matrices $A^{(i)}$ and $b^{(i)}$). Iteratively, we obtain in the $l$-th step convex polytopes $K_{m_l}$ and a polynomial $f_{m_1,\ldots,m_l}$ of degree $l$ such that
		\begin{align}
			|T_{rw}^{i}| = \sum_{m_1 = 1}^{2N} \dots \sum_{m_l = 1}^{2N} \int_{K_{m_l}} f_{m_1,\ldots,m_l}(x_{l + 1},...,x_n, rw) \ d(x_{l + 1},...,x_n).
		\end{align}
	The algorithm terminates after $n$ steps and leaves a polynomial in $r$ of degree $n$ as described in \eqref{eq:trapezium_formula}. The assertion now follows from \eqref{eq:trapezium_formula} by combining the
  above identities, i.e.
  \begin{align}
    C_u(rw)	%
       = |\Omega| -  \frac 12 \sum_{i = 1}^N |T_{rw}^{(i)}|	\ %
    & =\|u\|_{L^1(\T^n)} - \frac 12 \sum_{i = 1}^N \sum_{k = 1}^n d_k^{(i)}(w) r^k
  \end{align}
  for all $r \in (0,R)$. Averaging with respect to $w$ then yields
  \begin{align} \label{expr-pol} %
    c_u(r) = \|u\|_{L^1(\T^n)} + \sum_{j = 1}^n a_j r^j && \text{for all } 0
                                                           \leq r < R
  \end{align}
  for some $a_1,\ldots,a_n \in \R$.  Proposition \ref{prp-autocorrelation_bv}
  \ref{it-prp-autocorrelation_bv:perimeter} shows that
  $a_1 = - \frac{\omega_{n - 1}}{n \omega_n} P[u]$. By \eqref{expr-pol} we have
  $c''(0) = 2 a_2$. In view of Proposition \ref{prp-autocorrelation_bv}
  \ref{it-prp-autocorrelation_bv:perimeter} we have $c''(0) \geq 0$ and hence
  $a_2 \geq 0$.
  
  \begin{figure}[t] \begin{center} \includegraphics[scale=.2]{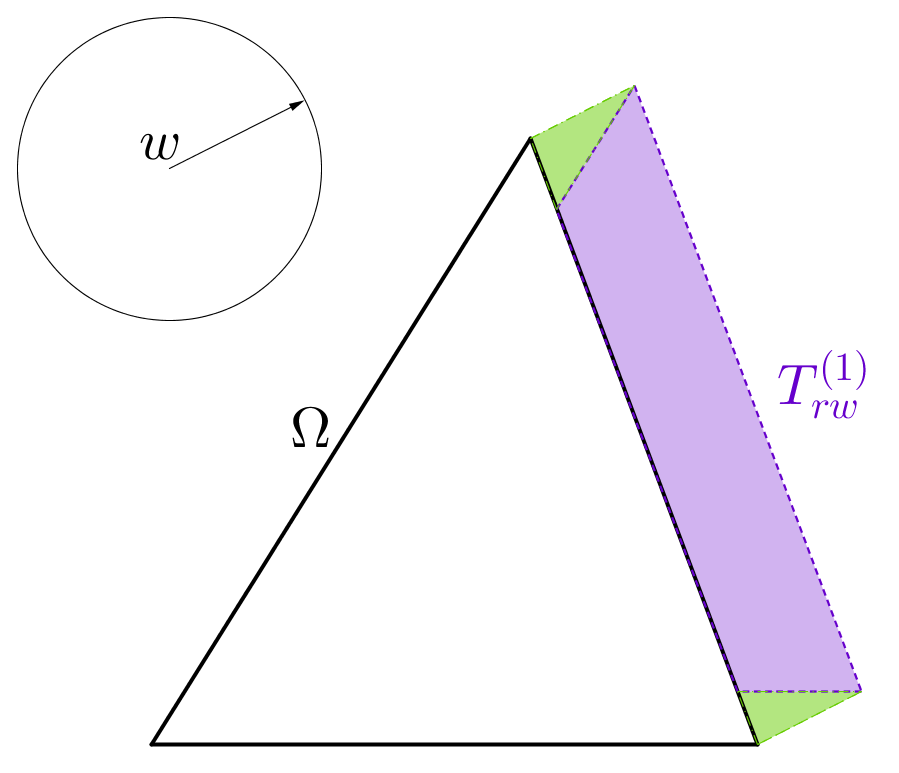}
      \caption{Sketch of the polynomial dependence on $r$ of $|T_{rw}^{(1)}|$. The height of the trapezium is given by $h = r|\nu \cdot w|$, where $\nu$ is the outward normal of the corresponding face. The volume of the green triangles is of order $r^2$.} \label{fig:trapezium_order}
    \end{center} \end{figure}
\end{proof}

\begin{remark}[Full space setting] %
	We note that with analogous arguments the result of Lemma
  \ref{lemma:polytope_autocorrelation} also holds in the full space
  setting.
\end{remark}

\section{Compactness and \texorpdfstring{$\Gamma$--convergence}{Gamma--convergence}} \label{sec:proofs}

\subsection{Reformulation of the energy}

For the subsequent analysis it is convenient to use the integrated kernel
\begin{align} \label{def-Phi} %
  \Phi(r) \ %
  := \ \sigma_n \int_r^\infty \rho^{n - 1} \KK(\rho) \ d\rho \ %
  = \ \int_{B_r(0)^c} \KK(z) \ dz && \text{for all } r > 0,
  \end{align}
  instead of the kernel itself.  We collect some basic properties of $\Phi$
  which follow from our assumptions on $\KK$:
\begin{proposition}[Basic properties of $\Phi$] \label{prp-kprop}
  Assume $\KK$ satisfies \hyp1 -- \hyp3. Let $\Phi: (0,\infty) \longrightarrow \R$
  be given by \eqref{def-Phi} and let $\Phi_\eps(r) := \frac 1 \eps \Phi(\frac r\eps)$.
  Then $\Phi \in L^1(0,\infty)$ with $\Phi \geq 0$. Furthermore,
  \begin{enumerate}
  \item \label{it-prp-limits} $\displaystyle \lim_{r\to \infty} r \Phi(r) %
    = \lim_{r\to 0} r \Phi(r) = 0$.
   \item \label{it-prp-kprop:l1}
     $\displaystyle \int_0^\infty \Phi(r) \ dr = \int_{\Rn} |z| \, \KK(z) \ dz < \infty$.
  \item \label{it-prp-kprop:moment}
    $\displaystyle \int_0^1 r^j \Phi_\eps(r) \ dr \xrightarrow{\eps \to 0} 0$ \quad for $j \geq 1$.
  \end{enumerate}
\end{proposition}
\begin{proof}
  By definition of $\Phi$ and \hyp3 we have $\Phi \geq 0$ and
    \begin{align} \label{kernel-properties-derivative} %
      \Phi'(r) \ = \ -\sigma_n r^{n - 1} \KK(r) \qquad \text{ for all $r > 0$}.
    \end{align}
   By definition of $\Phi$ and assumption \hyp2 we have
   \begin{align} \label{eq:Phi_upper_boundary} |R \Phi(R)| \leq \sigma_n R
     \int_R^\infty \rho^{n - 1} |\KK(\rho)| \ d\rho \ %
     \stackrel{\eqref{def-Phi}}\leq \ \sigma_n \int_R^\infty \rho^n |\KK(\rho)| \ d\rho
     \ %
     \to \ 0
   \end{align}
   as $R \to \infty$. Moreover, from Fatou's Lemma we obtain
   \begin{align} \label{eq:Phi_lower_boundary} %
     \limsup_{\varrho \to 0} \varrho \Phi(\varrho) %
     \leq \int_0^\infty \rho^{n - 1} |\KK(\rho)| \bigg( \limsup_{\varrho \to
     0} \varrho \chi_{(\varrho,\infty)}(\rho) \bigg) \ d\rho = 0,
   \end{align}
   which shows that \ref{it-prp-limits} holds. Integrating by parts we get
     for $0 < \varrho < R < \infty$
     \begin{align} \label{eq:Phi_integrable} %
       \sigma_n \int_\varrho^R r^n \KK(r) \ dr \ %
       &\stackrel{\eqref{kernel-properties-derivative}}{=} -\int_\varrho^R r \Phi'(r)
         \ dr = -r \, \Phi(r) \bigg|_\varrho^R + \int_\varrho^R \Phi(r) \ dr,
     \end{align}
     which yields \ref{it-prp-kprop:l1} in the limits $\varrho \to 0$ and
     $R \to \infty$ by \ref{it-prp-limits} and the Dominated Convergence
     Theorem.  Assertion \ref{it-prp-kprop:moment} holds since
     $\Phi \in L^1(0,\infty)$.
 \end{proof}
 We are ready to rewrite the nonlocal term in $E_{\gamma,\eps}$ in terms of the
 symmetrised autocorrelation function and the integrated kernel $\Phi$:
\begin{lemma}[Representation of energy by autocorrelation
  function] \label{lem:reformulation_non_local} Let $\gamma, \eps > 0$ and
  suppose $\KK$ satisfies \hyp1--\hyp3. Let $\Phi$ be given by \eqref{def-Phi} and let $\Phi_\eps(r) := \frac 1 \eps \Phi(\frac r\eps)$.
  Then
  \begin{align} \label{eq:represenatation_non_local} E_{\gamma,\eps}[u] =
    2\gamma_{\rm crit} \int_0^{\infty} \Phi_\eps(r) \,
    \bigg[\frac{\gamma}{\gamma_{\rm crit}} c'_u(r) - c'(0) \bigg] \ dr &&
    \text{for all } u \in \AA,
  \end{align}
  where $\gamma_{\rm crit}$ is represented in terms of $\Phi$ via
 \begin{align} \label{gam-Phi} %
   \gamma_{\rm crit} \ %
   = \ \frac{n \omega_n}{2\omega_{n - 1}}
   \|\Phi\|_{L^1(0,\infty)}^{-1}.
 \end{align}
\end{lemma}
\begin{proof}
  We first note that the representation \eqref{gam-Phi} follows from
    Proposition \ref{it-prp-kprop:l1} together with \eqref{def-gam}. Since $u$ only takes values in $\{0,1\}$ we have with Proposition
  \ref{prp-autocorrelation_bv} \ref{it-prp-autocorrelation_bv:origin} 
  \begin{align}
      \frac 1\eps \int_{\R^n} \KK_\eps(z) & \int_{\T^n} |u(x) - u(x + z)| \ dx \ dz	\\[6pt]       
                           & = \frac 1\eps \int_{\R^n} \KK_\eps(z) \int_{\T^n} |u(x) - u(x + z)|^2 \ dx \ dz	\\[6pt]
                           & = \frac 2\eps \int_{\R^n} \KK_\eps(z) \bigg[ \int_{\T^n} |u(x)|^2 \ dx - \int_{\T^n} u(x) \, u(x + z) \ dx \bigg] \ dz		\\[6pt]
                           & = \frac 2\eps \int_{\R^n} \KK_\eps(z) \, [ C_u(0) - C_u(z)] \ dz	\\[6pt]
                           & = \frac{2\sigma_n}{\eps^{n + 1}} \int_0^\infty r^{n - 1} \KK(\tfrac r\eps) \, [c_u(0) - c_u(r)] \ dr,
  \end{align}
  where in the last line we switched to polar coordinates. Let
  $0 < \varrho < R < \infty$. Noting that $c_u$ is Lipschitz (Proposition
  \ref{prp-autocorrelation_bv}), we integrate by parts to obtain
  \begin{align}
   \frac{2\sigma_n}{\eps^{n + 1}} & \int_\varrho^R r^{n - 1} \KK(\tfrac r\eps) \, [c_u(0) - c_u(r)] \ dr	\\[6pt]
   							& \kern -4pt\stackrel{\eqref{kernel-properties-derivative}}= -\frac{2}{\eps} \int_\varrho^R [\Phi(\tfrac r\eps)]' \, [c_u(0) - c_u(r)] \ dr	\\
    &= \ -2 \Phi_\eps(r) \, [c_u(0) -  c_u(r)] \bigg|_\varrho^R - 2 \int_\varrho^R \Phi_\eps(r) \, c'_u(r) \ dr,
      \label{eq:reformulation_integration_by_parts}
  \end{align}
  noting that $\Phi_\eps(r) = \frac 1\eps \Phi(\frac r\eps)$. Using Proposition \ref{prp-autocorrelation_bv}
  \ref{it-prp-autocorrelation_bv:perimeter} we obtain
  \begin{align}
    |\Phi_\eps(r) \, [c_u(0) -  c_u(r)]| \lesssim r \Phi_\eps(r)  \, P[u]	&& \text{for all } r > 0.
  \end{align}
  Thus the boundary terms in \eqref{eq:reformulation_integration_by_parts}
  vanish in the limit $\varrho \to 0$ and $R \to \infty$ by Proposition
  \ref{prp-kprop} \ref{it-prp-limits}. It follows from the Dominated Convergence
  Theorem and the computations above that
  \begin{align} \label{eq:non_local_term_Phi_representation} 
  	\frac 1\eps  \int_{\R^n} \KK_\eps(z) \int_{\T^n} |u(x) - u(x + z)| \ dx =
    -2\int_0^\infty \Phi_\eps(r) \, c'_u(r) \ dr.
  \end{align}
  The claim follows by inserting \eqref{eq:non_local_term_Phi_representation}
  into \eqref{eq:definition_energy} and using Proposition \ref{prp-autocorrelation_bv}
  \ref{it-prp-autocorrelation_bv:perimeter} to rewrite $P[u]$ in terms of
  $c'_u(0)$.
\end{proof}

\subsection{Proofs of the main theorems} \label{subsec:proofs}

We next use the reformulation of the energy $E_{\gamma,\eps}$ from Lemma
\ref{lem:reformulation_non_local} to show that the limit energy is a pointwise
lower bound. This will simplify the proof of the compactness in Theorem
\ref{thm:compactness} and the liminf inequality of Theorem
\ref{thm:convergence}. It relies on the bound
$c_u'(0) = \min_{r \geq 0} c_u'(r)$ for all $u \in \AA$ (Prop.
\ref{prp-autocorrelation_bv} \ref{it-prp-autocorrelation_bv:perimeter}).

\begin{proposition}[Lower bound for the
  energy] \label{prp-pointwise_bound_energy} Let $\gamma, \eps > 0$ and suppose
  that $\KK$ satisfies \hyp1--\hyp3. Then
  \begin{align} \label{eq:pointwise_lower_bound}
    E_{\gamma,0}[u] \leq E_{\gamma,\eps}[u]	&& \text{for all } u \in \AA.
  \end{align}
\end{proposition}

\begin{proof}
  By Proposition \ref{prp-kprop} \ref{it-prp-kprop:moment}, Proposition
  \ref{prp-autocorrelation_bv} \ref{it-prp-autocorrelation_bv:perimeter} and
  Lemma \ref{lem:reformulation_non_local} we obtain
  \begin{align}
    E_{\gamma,0}[u] \ %
    &\stackrel{\eqref{gam-Phi}}{=} 2\gamma_{\rm crit} \int_0^{\infty} \Phi_\eps(r) \, \bigg[ \frac{\gamma}{\gamma_{\rm crit}} c'_u(0) - c_u'(0) \bigg] \ dr\\[6pt]
    &\leq \ 2\gamma_{\rm crit} \int_0^{\infty} \Phi_\eps(r) \, \bigg[\frac{\gamma}{\gamma_{\rm crit}} c'_u(r) - c_u'(0) \bigg] \ dr \ = \ E_{\gamma,\eps}[u].
  \end{align}
\end{proof}
The next proposition is a convergence result and simplifies the
non--compact-ness (Theorem \ref{thm:compactness}
\ref{it-thm:compactness:non_compactness}) and limsup inequality of Theorem
\ref{thm:convergence}. The key observation is that the family of integral
kernels $\|\Phi\|_{L^1(\R_+)}^{-1} \Phi_\eps$ forms an approximation of the
identity, so we formally achieve with the reformulation of the energy (see Lemma
\ref{lem:reformulation_non_local})
	\begin{align}
		E_{\gamma,\eps}[u]	& = 2\gamma_{\rm crit} \int_0^{\infty} \Phi_\eps(r) \, \bigg[\frac{\gamma}{\gamma_{\rm crit}} c'_u(r) - c'(0) \bigg] \ dr	\\[6pt]
							& \xrightarrow{\eps \to 0} 2\gamma_{\rm crit} \| \Phi\|_{L^1(0,\infty)} \, \langle\delta_0, \tfrac{\gamma}{\gamma_{\rm crit}} c'_u(r) - c'(0)  \rangle = E_{\gamma,0}[u].
	\end{align}
        However, for general functions $u \in \AA$ it is not clear if the
        convergence $E_{\gamma,\eps}[u] \to E_{\gamma,0}[u]$ holds pointwise
        since the function $c_u'$ is not differentiable at the origin in
        general. We hence consider the smaller class of polytopal domains.  This
        is where we use the formula for the autocorrelation function for
        polytopes (see Lemma \ref{lemma:polytope_autocorrelation}).
        
\begin{proposition}[Pointwise convergence for polytopes] \label{prp-energy_convergence_polytopes}
	Let $\gamma > 0$ and suppose $\KK$ satisfies \hyp1--\hyp3. Let $\Omega \subset \T^n$ be a polytope and let $u = \chi_\Omega$. Then
		\begin{align}
                  E_{\gamma,\eps}[u] \xrightarrow{\eps \to 0} E_{\gamma,0}[u].
		\end{align}
\end{proposition}

\begin{proof}
	Let $\eps > 0$. It suffices to show that the non--local term converges, i.e. (see Lemma \ref{lem:reformulation_non_local} and \eqref{eq:non_local_term_Phi_representation})
		\begin{align} \label{eq:non-local_convergence}
			\int_0^\infty & \Phi_\eps(r) \, c'_u(r) \ dr \xrightarrow{\eps \to 0} -\frac{1}{2\gamma_{\rm crit}} P[u].
		\end{align}
	Lemma \ref{lemma:polytope_autocorrelation} yields the existence of $a_1,\ldots, a_{n - 1} \in \R$ and $0 < R < 1$ such that 
		\begin{align}
			c'_u(r) = -\frac{\omega_{n - 1}}{n \omega_n} P[u] + \sum_{j = 1}^{n - 1} a_j \, r^j	&& \text{for all } 0 < r < R.
		\end{align}
	We use this identity to compute
				\small{\begin{align}
					\int_0^\infty & \Phi_\eps(r) \, c'_u(r) \ dr	\\
													 = & -\frac{\omega_{n - 1}}{n \omega_n} P[u] \int_0^R \Phi_\eps(r) \ dr + \sum_{j = 1}^{n - 1} a_j \int_0^R r^j \Phi_\eps(r) \ dr + \int_R^\infty \Phi_\eps(r) \, c'_u(r) \ dr	\\[6pt]
													 = & -\frac{\omega_{n - 1}}{n \omega_n} P[u] \, I_1^{(\eps)} + I_2^{(\eps)} + I_3^{(\eps)}.
				\end{align}}
                              For the first integral, we observe that by monotone convergence 
                        \begin{align}
					I_1^{(\eps)}	& =  \int_0^{\frac R \eps} \Phi(r) \ dr \xrightarrow{\eps \to 0} \|\Phi\|_{L^1(0,\infty)}.
				\end{align}
                                For the second integral we use Proposition
                                \ref{prp-kprop} \ref{it-prp-kprop:moment} to
                                obtain
				\begin{align}
					|I_2^{(\eps)}|	& \leq \sum_{j = 1}^{n - 1} |a_j| \int_0^1 r^j \Phi_\eps(r) \ dr \xrightarrow{\eps \to 0} 0.
				\end{align}
			For the third integral, using Proposition \ref{prp-autocorrelation_bv} \ref{it-prp-autocorrelation_bv:perimeter} we have
				\begin{align}
					|I_3^{(\eps)}|	& \leq \int_R^\infty \Phi_\eps(r) \ |c_u'(r)| \ dr \leq \frac{\omega_{n - 1}}{n \omega_n} P[u] \int_{\frac R \eps}^\infty \Phi(r) \ dr \xrightarrow{\eps \to 0} 0.
				\end{align}
			Altogether \eqref{eq:non-local_convergence} follows.
\end{proof}

We are now in the position to prove the compactness (Theorem \ref{thm:compactness} \ref{it-thm:compactness:compactness}) and non--compactness (Theorem \ref{thm:compactness} \ref{it-thm:compactness:non_compactness}) of $E_{\gamma,\eps}$:

\begin{proof}[Proof of Theorem \ref{thm:compactness}] \text{}
	In order to prove assertion \ref{it-thm:compactness:compactness} we use Proposition \ref{prp-pointwise_bound_energy} to obtain
		\begin{align}
			0 \leq \Big(1 - \frac{\gamma}{\gamma_{\rm crit}} \Big) P[u_\eps] = E_{\gamma,0}[u_\eps] & \leq E_{\gamma,\eps}[u_\eps] \lesssim 1
		\end{align}
uniformly in $\eps$. Using the $L^1$--compactness of the perimeter functional proves the claim. In order to prove assertion \ref{it-thm:compactness:non_compactness} we construct a sequence of laminates with the desired properties. We define $\Omega := (0,\theta) \times [0,1)^{n - 1} \subset \T^n$ and $v_k(x) := \chi_{\Omega}(kx)$. Then $P[v_k] = 2^{k + 1}$ and $v_k \in \AA$ for all $k \in \N$. Moreover, there does not exist a subsequence of $v_k$ which converges in $L^1$. Indeed, assume there exists a subsequence (not relabeled) and $v \in L^1(\T^n,\{0,1\})$, such that $v_k \to v$ in $L^1$ as $k \to \infty$. Then $\|v\|_{L^1(\T^n)} = \theta$. We note that $v_k \rightharpoonup^* \theta$ in $L^\infty$ as $k \to \infty$ (see \cite[Example 2.7]{BraDef:1998}) and thus
	\begin{align}
		\|v_k - v\|_{L^1(\T^n)}	& = \int_{\T^n} v_k(x) \ dx + \int_{\T^n} v(x) \ dx - 2\int_{\T^n} v_k(x) \, v(x) \ dx	\\[6pt]
								& \xrightarrow{k \to \infty} 2 \theta - 2\theta^2 \neq 0,
	\end{align}
a contradiction. Further, from Proposition \ref{prp-energy_convergence_polytopes} we obtain
		\begin{align}
			E_{\gamma,\eps}[v_k] \xrightarrow{\eps \to 0} \Big(1 - \frac{\gamma}{\gamma_{\rm crit}} \Big) P[v_k] \leq 0,
		\end{align}
	since $\gamma \geq \gamma_{\rm crit}$. Hence, for all $k \in \N$ there exists $\eps^*(k)$ such that
  	\begin{align}\label{eq:noncompact_approx}
  		|E_{\gamma,\eps}[v_k] - E_{\gamma,0}[v_k]| < \frac 1k 	&& \text{for all } 0 < \eps < \eps^*.
  	\end{align}
  	For $\frac 12 \min\{\frac{1}{k + 1}, \eps^*(k + 1)\} \leq \eps \leq \frac 12 \min\{\frac 1k, \eps^*(k)\}$, we choose $u_\eps = v_k$. Then there does not exist an $L^1$--convergent subsequence as was shown above, and
  		\begin{align} \label{eq:supercritical_estimate}
  			E_{\gamma,\eps}[u_\eps] \stackrel{\eqref{eq:noncompact_approx}}\leq E_{\gamma,0}[v_k] + \frac 1k \leq 3,
  		\end{align}
  	from which the claim follows.
\end{proof}

From the techniques used to prove Theorem \ref{thm:compactness}, we are now able to show Corollary \ref{coro:finer_behaviour}:

\begin{proof}[Proof of Corollary \ref{coro:finer_behaviour}]
  Assertion \ref{it-coro:finer_behaviour1} follows directly from Proposition
  \ref{prp-pointwise_bound_energy}. In order to show assertion \ref{it-coro:finer_behaviour2}, let $v_k$ and $u_\eps$ as in the proof of Theorem \ref{thm:compactness} \ref{it-thm:compactness:non_compactness}. Then 
  		\begin{align}
  			P[u_\eps] \xrightarrow{\eps \to 0} +\infty, && E_{\gamma,\eps}[u_\eps] \stackrel{\eqref{eq:supercritical_estimate}}\leq E_{\gamma,0}[v_k] + \frac 1k,
  		\end{align}
  	from which the claim follows since $\gamma > \gamma_{\rm crit}$.  
  In order to prove the
  necessity of a sequence with unbounded perimeter, we argue by contradiction:
  assume that (after selection of a subsequence) $P[u_\eps] \lesssim 1$, then
  $-1 \lesssim E_{\gamma,\eps}[u_\eps]$ by Proposition
  \ref{prp-pointwise_bound_energy}, which contradicts
  \eqref{eq:sequence_energy_negative_infinity}.
\end{proof}

We are now in the position to prove the $\Gamma$--convergence of $E_{\gamma,\eps}$:

\begin{proof}[Proof of Theorem \ref{thm:convergence}]
  For the liminf inequality, let $u_\eps \to u$ in $L^1(\T^n)$. By Theorem
  \ref{thm:compactness} \ref{it-thm:compactness:compactness} we can assume
  without loss of generality that $u_\eps, u \in \AA$. From the lower
  semicontinuity of the perimeter functional and Proposition
  \ref{prp-pointwise_bound_energy} we obtain
  \begin{align}
    0 \leq E_{\gamma,0}[u]	& \leq \liminf_{\eps \to 0} E_{\gamma,0}[u_\eps] \stackrel{\eqref{eq:pointwise_lower_bound}}{\leq} \liminf_{\eps \to 0} E_{\gamma,\eps}[u_\eps].
  \end{align}
  To show the limsup inequality, we use Proposition
  \ref{prp-energy_convergence_polytopes} and an approximation using polytopes. If
  $u \notin \AA$, the statement is obvious. So let $u \in \AA$. Then there exists a sequence of polytopes $\Omega_k \subset \T^n$, such that (see \eqref{eq:density})
  	\begin{align}\label{eq:upper_bound_polytope_approx}
  		\|\chi_{\Omega_k} - u\|_{L^1(\T^n)} \xrightarrow{k \to \infty} 0, && |P[\chi_{\Omega_k}] - P[u]| \xrightarrow{k \to \infty} 0.
  	\end{align}
  By a rescaling of $\Omega_k$, we can without loss of generality assume that $v_k := \chi_{\Omega_k} \in \AA$. We proceed similarly as in the proof of Theorem \ref{thm:compactness}. By Proposition \ref{prp-energy_convergence_polytopes} we have $E_{\gamma,\eps}[v_k] \to E_{\gamma,0}[v_k]$ as $\eps \to 0$ for any $k \in \N$. Hence, for all $k \in \N$ there exists $\eps^*(k)$ such that
  	\begin{align}\label{eq:limsup_approx}
  		|E_{\gamma,\eps}[v_k] - E_{\gamma,0}[v_k]| < \frac 1k 	&& \text{for all } 0 < \eps < \eps^*.
  	\end{align}
  	For $\frac 12 \min\{\frac{1}{k + 1}, \eps^*(k + 1)\} < \eps < \frac 12 \min\{\frac 1k, \eps^*(k)\}$ we choose $u_\eps = v_k$. Then 
  		\begin{align}
  			\|u_\eps - u\|_{L^1(\T^n)} \xrightarrow{\eps \to 0} 0,	&& |P[u_\eps] - P[u]| \xrightarrow{\eps \to 0} 0, && E_{\gamma,\eps}[u_\eps] \stackrel{\eqref{eq:limsup_approx}}\leq E_{\gamma,0}[v_k] + \frac 1k,
  		\end{align}
  	from which the claim follows.
\end{proof}

	\renewcommand{\thesection}{\Alph{section}}
	\setcounter{section}{0}

      \section{Appendix} \label{sec:appendix}

     In the following Lemma, we present some examples of admissible kernels.
      \begin{lemma} \label{lem:helmholtz} Let $s \in (1,2]$ and let $\KK^{(s)} \in L^1(\R^n)$ be the unique
        solution to
        \begin{align}\label{eq:s-helmholtz}
          \KK^{(s)} + (-\Delta)^{\frac s2} \KK^{(s)} = \delta_0 &&
                                         \text{in } \R^n.
        \end{align}
        Then $\KK^{(s)}$ satisfies \hyp1 -- \hyp3. Moreover, for $s = 2$ we have
        \begin{align} \label{eq:lemma:helmholtz_crit} \int_\Rn |z| \, \KK^{(2)}(z) \
          dz = \sqrt \pi \, \Gamma(\tfrac{n + 1}2) \, \Gamma(\tfrac{n}2)^{-1}
        \end{align}
        and the associated critical value $\gamma_{\rm crit}^{(2)}$
        (cf. \eqref{def-gam}) is given by $\gamma_{\rm crit}^{(2)} = 1$.
      \end{lemma}
      \begin{proof}
  We first note that the equation can be solved explicitly with Fourier
  methods. More precisely it holds
  \begin{align}
    \hat \KK^{(s)}(\xi) := \int_{\R^n} \KK^{(s)}(z) \, e^{-2\pi \mathbf i z \cdot \xi} \ dz = \frac{1}{1 + (2\pi)^s |\xi|^s}	&& \text{for all } \xi \in \R^n.
  \end{align}
  Since the Fourier transform is an injective map $\hat \bullet: L^1(\Rn) \longrightarrow C^0_b(\R^n)$, the solution of \eqref{eq:s-helmholtz} is unique. Since $\hat \KK^{(s)}$ is radial, so is $\KK^{(s)}$ and thus \hyp1. We treat the cases $s = 2$ and $s \neq 2$
  separately. If $1 < s < 2$, we get from \cite[Lemma 1.2]{MelWu:2021} that
  $\KK^{(s)} \geq 0$ and $|\KK^{(s)}(r)| \lesssim r^{-n - s}$ for $r \geq 1$
  which directly yield both \hyp2 and \hyp3. If $s = 2$, switching to polar
  coordinates in \eqref{eq:s-helmholtz}, $\KK^{(2)}$ solves
  \begin{align}
    \KK^{(2)}(r) - \frac{n - 1}{r} (\KK^{(2)})'(r) - (\KK^{(2)})''(r) = 0	&& \text{for all } r > 0.
  \end{align}
  Substituting $\KK^{(2)}(r) := r^{1 - \frac n2} g(r)$, we obtain
  \begin{align}
    r^2 g''(r) + r  g'(r) - (r^2 + (\tfrac n2 - 1)^2) g(r) = 0	&& \text{for all } r > 0.
  \end{align}
  This is the modified Bessel equation. The unique decaying solution is given by
  $g = K_{\frac n2 - 1}$, where $K_\nu$ denotes the decaying modified Bessel
  function of the second kind of genus $\nu \geq 0$, i.e.
  \begin{align} \label{eq:bessel_ode} K_\nu(r) = \int_0^\infty \cosh(\nu t) \,
    e^{-r \, \cosh(t)} dt && \text{for all } r > 0.
  \end{align}
  Altogether, we obtain
  \begin{align}
    \KK^{(2)}(z) = a |z|^{1 - \frac n2} K_{\frac n2 - 1}(|z|)	&& \text{for all } z \in \R^n,
  \end{align}
  where $a = (2\pi)^{-\frac n2}$ such that $\|\KK^{(2)}\|_{L^1(\R^n)} = 1$. Since
  $K_\nu \geq 0$ for all $\nu \geq 0$, it follows \hyp3. The identity
  \eqref{eq:lemma:helmholtz_crit} and \hyp2 follow from
  \cite[p.676]{GraRyz:2015}.
\end{proof}

\renewcommand{\em}{\it} \bibliographystyle{plain} \bibliography{biblio} \normalsize

\end{document}